\documentclass[11pt,reqno]{amsart}

\usepackage{amssymb}
\usepackage{latexsym}

\newtheorem{theorem}{Theorem}
\newtheorem{lemma}[theorem]{Lemma}
\newtheorem{corollary}[theorem]{Corollary}

\newtheorem{conj}[theorem]{Conjecture}
\newtheorem{rem}[theorem]{Remark}

\newcommand{\Z}{{\mathbb Z}}
\newcommand{\N}{{\mathbb N}}

\newcommand{\beeq}{\begin{eqnarray*}}
\newcommand{\eneq}{\end{eqnarray*}}

\newcommand{\be}{\begin{equation}}
\newcommand{\ee}{\end{equation}}

\author{J. Barajas and  O. Serra}
\thanks{Research supported
by the Spanish Ministry of Eduaction under grant
MTM2005-08990-C02-01 and by the Catalan Research Council under grant
2005SGR00256}
\address{ Dept. of Applied Mathematics IV\\ Polytechnical
University of Catalonia, Barcelona,
Spain\\\{jbarajas,oserra\}@ma4.upc.edu}

\title{The lonely runner with seven runners}

\begin{document}
\baselineskip 16pt

\maketitle

\begin{abstract}
 Suppose $k+1$ runners having nonzero constant speeds run laps on a
unit-length circular track starting at the same time and place. A
runner is said to be lonely if she is at distance at least $1/(k+1)$
along the track to every other runner. The lonely runner conjecture
states that every runner gets lonely. The conjecture has been proved
up to six runners ($k\le 5$). A formulation of the problem is
related to the regular chromatic number of distance graphs. We use a
new tool developed in this context to solve the first open case of
the conjecture with seven runners.
\end{abstract}

\maketitle

\section{Introduction}

Consider $k+1$ runners on a unit length circular track. All the
runners start at the same time and place and each runner has a
constant speed. A runner is said to be lonely at some time if she is
at distance at least $1/(k + 1)$ along the track from every other
runner. The {\it Lonely Runner Conjecture} states that each runner
gets lonely. The Lonely Runner Conjecture has been introduced by
Wills \cite{wills} and independently by Cusick \cite{C1}, and it has
been given this pitturesque name by Goddyn \cite{BGGST}. For $k=3$,
there are four proofs in the context of diophantine approximations:
Betke and Wills \cite{BW} and Cusick \cite{C1, C2, C3}. The case
$k=4$ was first proved by Cusick and Pomerance \cite{CP}, with a
proof requiring  computer checking. Later, Bienia et al.
\cite{BGGST} gave a simpler proof for the case $k=4$. The case $k =
5$ was proved by Bohman, Holzman and Kleitman \cite{BHK}. A simpler
proof for this case was given later by Renault \cite{R}.

This problem appears in   different contexts. Cusick \cite{C1} was
motivated by an application in  view obstruction problems in
$n$--dimensional geometry, and Wills \cite{wills} considered the
problem from the diophantine approximation point of view. Biennia et
al. \cite{BGGST} observed that the solution of the lonely runner
problem implies a theorem on nowhere zero flows in regular matroids.
Zhu \cite{Z} used  known results for the lonely runner problem to
compute the chromatic number of distance graphs. In \cite{BS2} a
similar approach was used to study the chromatic number of circulant
graphs.

A convenient and usual reformulation of the lonely runner conjecture
can be obtained by assuming that all speeds are integers, not all
divisible by the same prime, (see e.g. \cite{BHK}) and that the
runner to be lonely has zero speed. Let $\| x\| $ denote the
distance of the real number $x$ to its nearest integer. In this
formulation the Lonely Runner Conjecture states that, for any set
$D$ of $k$ positive integers, there is a real number $t$ such that
$\| td\|\ge 1/(k+1)$ for each $d\in D$. We shall consider a discrete
version of the lonely runner problem.

Let $N$ be a positive integer. For an integer $x\in \Z$ we denote by
$|x|_N$ the residue class of $x$ or $-x$ in the interval $[0,N/2]$.
For a set $D\subset \N$ of positive integers we  define the
 {\em regular chromatic number}  $\chi_r(N,D)$
as
$$
\chi_r(N,D)=\min\{k:\exists \lambda \in \Z_N \mbox{ such that }
 |\lambda d|_N \ge \frac{N}{k} \mbox{ for each } d\in D\},
$$
if $D$ contains no multiples of $N$ and $\chi_r (N,D)=\infty$
otherwise. We define the {\em regular chromatic number} of $D$ as
$$
\chi_r (D)= \liminf_{N\rightarrow\infty} \chi_r(N,D).
$$

The reason for calling chromatic numbers the parameters defined
above stems from applications of the lonely runner problem to the
study of the chromatic numbers of distance graphs and circulant
graphs; see e.g \cite{BS1,BS2,Z}. In this terminology, the lonely
runner conjecture can be equivalently formulated as follows.

\begin{conj}\label{conj1} For every set $D\subset \Z$   of positive
integers with $\gcd (D)=1$,
$$
\chi_{r}(D)\le |D|+1.
$$
\end{conj}

In \cite{BS1} the so--called {\it Prime Filtering Lemma} was
introduced as a tool to  obtain a  characterization of sets $D$ with
$|D|=4$ for which equality holds in Conjecture \ref{conj1}. The
Prime Filtering Lemma provides a short proof of the conjecture for
$|D|=4$ (five runners) which we include in Section \ref{sec:d=4}
just to illustrate the technique.  In Section \ref{sec:pf}  we
formulate a generalization of the lemma   and we then use it in the
rest of the paper  to solve the first open case of the conjecture
when $|D|=6$. As it will become clear in the coming sections, the
Prime Filtering Lemma essentially reduces the proof to a finite
problem in $\Z_7$ which can be seen as a generalization of the
Lonely Runner Problem in which the runners may have different
starting points. Unfortunately the conjecture does not always hold
in this new context and we have to proceed with a more detailed case
analysis.

\section{Notation and Preliminary results}\label{sec:pf}

For a positive integer $x$ and a prime $p$, the $p$--adic valuation of $x$ is
$$
\nu_p(x)=\max \{ k:\; x\equiv 0 \pmod {p^k}\}.
$$
We also denote by $r_p(x)=(xp^{-\nu_p(x)})_p$ the congruence class
modulo $p$ of the least coefficient in the $p$-ary expansion of $x$.

We shall consider the discrete version of the lonely runner problem
mostly in the integers modulo $N$ with $N$ a prime power. We denote
by $(x)_N$ the residue class of $x$ modulo $N$ in $\{ 0,1,\ldots
,N-1\}$ and we denote by $|x|_N$ the residue class of $x$ or $-x$
modulo $N$ in $\{ 0,1,\ldots , \lfloor N/2\rfloor \}$.

Let $D$ be a set of positive integers, let  $m=\max \nu_p (D)$ and
set $N=p^{m+1}$. Note that, for each $x\in \Z$,
$\nu_p(x)=\nu_p((x)_N)=\nu_p(|x|_N)$. By abuse of notation we still
denote by $D$ the set $\{ (d)_N:\; d\in D\}$ as a subset of  $\Z_N$
whenever the ambient group is clear from the context. The {\em
$p$--levels} of $D$ are
$$
D_p(i)=\{ d\in D:\; \nu_p(d)=i\}.
$$
Let $q=q_{p,m}:\Z\rightarrow\Z_p$ be defined as
$$
q(x)=(\lfloor \frac{x}{p^m}\rfloor)_p,
$$
that is, $q(x)=k$ is equivalent to $(x)_N\in [k(\frac{N}{p}),
(k+1)\frac{N}{p})$. We call the interval $[k(\frac{N}{p}),
(k+1)\frac{N}{p})$ the $k$-th $(N/p)$--{\em arc}. Our goal is to
find a multiplier $\lambda$ for $D'=D\setminus D_p(m)$ such that
\begin{equation}\label{eq:goal}
q(\lambda\cdot D')\cap \{0,p-1\}=\emptyset,
\end{equation}
where $\lambda\cdot X=\{ \lambda x; x\in X\}$.  Indeed, if
(\ref{eq:goal}) holds, then $|\lambda d|_N\ge N/p$ for each $d\in D$
and  $\chi_r (D)\le \chi_r (N,D)\le p$, giving Conjecture
\ref{conj1} whenever $|D|\ge p-1$.

We shall mostly use multipliers of the form $1+p^{m-j}k$. Let
    $$
    \Lambda_{j,p}=\{1+p^{m-j}k,\; 0\le k\le p-1\},\;\; j=0,1,\ldots ,m-1,
    $$
and
    $$
    \Lambda_{m,p}=\{1,2,\ldots ,p-1\}.
    $$
Note that all elements in $U\Z_N$, the multiplicative group of
invertible elements in $\Z_N$, can be obtained  as a product of
elements in $\Lambda_{0,p}\cup\Lambda_{1,p}\cup\cdots
\Lambda_{m,p}$. In what follows, by a multiplier we shall always
mean an invertible element in $\Z_N$ where $N$ is a  prime power for
some specified prime $p$.

For each $j$ and each $\lambda\in \Lambda_{j,p}$, we have
\begin{equation}\label{eq:stab}\nu_p (\lambda\cdot x)=\nu_p (x)\end{equation}
 and, if $\lambda=1+kp^{m-j}$, then
    \begin{equation}\label{eq:mult1}
    q(\lambda\cdot
    x)=\left\{
    \begin{array}{ll}
        q(x), & \hbox{ if } \nu_p(x)>j, \\
        q(x)+kr_p(x), & \hbox{ if } \nu_p(x)=j.
    \end{array}
    \right.
\end{equation}
In view of (\ref{eq:mult1}), when using a multiplier $\lambda\in
\Lambda_{j,p}$, the values of $q$ on the elements in the $p$--levels
$D_p(i)$ of $D$ with $i>j$ remain unchanged. The following result is
based in this simple principle. It gives a sufficient condition for
the existence of a multiplier $\lambda$ such that multiplication by
$\lambda$ sends every element $d\in D$ outside  a `forbidden' set
for $d$.

\begin{lemma}[Prime Filtering]\label{lem:pf} Let $p$ be a prime and let $D$ be
a set of positive integers.
Set    $m=\max \{\nu_p (d):\; d\in D\}$ and $N=p^{m+1}$.
 For each $d\in D$ let $F_d\subset \Z_p$. Suppose that
\begin{eqnarray*}
\sum_{d\in D_p(j)}|F_d|&\le &p-1 \mbox{  for each } j=0,1,\ldots ,m-1, \mbox{and} \\
\sum_{d\in D_p(m)}|F_d|&\le& p-2,
\end{eqnarray*}
Then there is a multiplier $\lambda$
 such that, for each $d\in D$,
$$q(\lambda d)\not\in F_d.$$
\end{lemma}

\begin{proof} For each $d \in D(m)$ we have
$q(\Lambda_{m,p} \cdot d)=\Lambda_{m,p}$. Hence there are at most
$|F_d|$ elements $\lambda$ in $\Lambda_{m,p}$ such that $q(\lambda
d)\in F_d$. Since  $\sum_{d\in D_p(m)}|F_d|\le p-2$, there is
$\lambda\in \Lambda_{m,p}$ such that $q(\lambda d)\not\in F_d$ for
each $d\in D_p(m)$.

Denote by $E(i) = \cup_{j \ge i} D(j)$. Let $r$ be the smallest
nonnegative integer $i$ for which  there is  some $\lambda_i\in
\prod_{j=i}^m\Lambda_{j,p}$  verifying $q(\lambda_i d) \not\in F_d$
for every $d\in E(i)$. We have seen that $r\le m$.

Suppose that $r>0$ and let $\lambda \in \Lambda_{r-1,p}$. It follows
from (\ref{eq:stab}) and (\ref{eq:mult1}) that, for each $d \in
E(r)$, we have $(\lambda \lambda_r d)_N= (\lambda_r d)_N$. Note also
that, for each $d \in D(r-1)$, we have $q(\lambda_r d \cdot
\Lambda_{r-1,p}) = \Lambda_{m,p}$. Hence there are at most   $|F_d|$
elements $\lambda$ in $\Lambda_{r-1,p}$ for which $q(\lambda
\lambda_r d) \in F_d$. Since $\sum_{d\in D_p(r-1)}|F_d| < p$ there
is  at least one $\lambda \in \Lambda_{r-1,p}$ for which
$\lambda\lambda_r d \not\in F_d$ for each $d\in E(r-1)$
contradicting the minimality of $r$. Thus $r=0$ and we are done.
\end{proof}

 We shall often use the following form  of   Lemma \ref{lem:pf}, in
which all  forbidden sets are the $0$-th and $(p-1)$--th
$(N/p)$--arcs.

\begin{corollary}\label{cor:pf1} With the notation of Lemma
\ref{lem:pf}, suppose that $|d|_N\ge N/p$ for each $d\in D_p(i)$ and
each $i\ge i_0$ for some positive integer $i_0\le m$. If
$$
|D_p(j)|\le \frac{(p-1)}{2}, \; j=0,1,\ldots ,i_0-1,
$$
then $$\chi_r (N,D)\le p.$$
\end{corollary}

\begin{proof}
Let $F_d=\{ 0,p-1\}$ for each $d\in D\setminus D_p(m)$. We can apply
Lemma \ref{lem:pf} to each element of $D'=D\setminus (\cup_{i\ge
i_0} D_p(i))$ since $\sum_{d\in D_p(j)}|F_d| = 2 |D_p(j)|\le (p-1) $
for each $j< i_0$. Thus there is $\lambda\in
\prod_{j<i_0}\Lambda_{j,p}$ such that $q(\lambda d)\not\in \{
0,p-1\}$ for each $d\in D'$. With such $\lambda$ we also have
$|\lambda d|_N=|d|_N\ge N/p$ for each
  $d\in D\setminus D'$. Hence the inequality $|\lambda d|_N\ge N/p$
  holds for each $d \in D$, which is equivalent to $\chi_r (N, D)\le
  p$.
\end{proof}

\section{The case with three and five runners}\label{sec:d=4}

Let us show that the cases with three ($|D|=2$) and five ($|D|=4$)
runners can be easily handled. In other words, we prove in a simple
way that $\chi_r(D) \le |D|+1$ for those sets with $|D| = 2$ or
$|D|=4$.

For $|D|=2$, either the two elements in $D$ are relatively prime
with $3$ or they have different $3$-adic valuations. In both cases
Corollary \ref{cor:pf1} with $p=3$ applies and we get $\chi_r (D)\le
3$.

Suppose now that $|D|=4$. Let $m=\max \nu_5(D)$ and $N=5^{m+1}$.
Since we assume that $\gcd (D)=1$, we always have $D_5(0)\neq
\emptyset$. By definition we have $D_5(m)\neq \emptyset$ as well. If
 $|D_5(i)|\le 2$ for each $i<m$   then we are done by
Corollary \ref{cor:pf1}. Therefore we only have to consider the case
$|D_5(0)|= 3$ and $|D_5 (m)|=1$.

Put $A=D_5(0)=\{ d_1, d_2, d_3\}$ and $D_5(m)=\{ d_4\}$. We shall
show that, up to multiplication of elements in $\Lambda_{0,5}\cup
\Lambda_{m,5}$, we have $q(A)\cap \{0,4\}=\emptyset$. Since these
multiplications preserve the inequality $|d_4|_N\ge N/5$ we will
conclude that $\chi_r (D)\le 5$.

Let $d\in A$. For each $\lambda_k=1+k5^m\in \Lambda_{0,5}$  we have
    \be\label{eq:lj}
    q(\lambda_k d)=q(d)+kr_5(d),
    \ee
and for each $j\in \{ 1,2,3\}\subset \Lambda_{m,5}$,
    \be\label{eq:j}
    q((j+1) d)\subset q(jd)+q(d)+\{ 0,1\}\subset (j+1)q(d)+\{
    0,1,\ldots ,j\}.
    \ee

Since we can replace each $d\in A$ by $-d$ we may assume that  all
elements in $A$ belong to two nonzero congruence classes modulo $5$,
say $(A)_5 \subset \{ 1,2\}$. Let $A_s=\{ d\in A:\; (d)_5=s\} $,
$s\in \{1,2\}$, denote the most popular congruence class.

Let us denote by $\ell (A)$ the cardinality of the smallest
arithmetic progression of difference one in $\Z_5$ which contains
$q(A)$. Let us show that
    \be\label{eq:kk+1}
    \ell (j\lambda_k\cdot A)\le |A_s|
    \ee
for some  $j\in \Lambda_{m,5}$ and some $\lambda_k\in
\Lambda_{0,5}$.

Suppose that $|A_s|=3$ and assume that (\ref{eq:kk+1}) does not hold
 for $j=1$. By (\ref{eq:lj}) we may assume, up to multiplication by
some $\lambda_k$, that $q(d_1)=0$, $q(d_2)=2$ and $q(d_3)=3$.

By (\ref{eq:j}) we have $q(2d_1)\in \{0,1\}$, $q(2d_2)\in \{0,4\}$
and $q(2d_3)\in \{1,2\}$. If (\ref{eq:kk+1}) does not hold for $j=2$
either, then $q(2d_2)=4$ and $q(2d_3)=2$. Again by (\ref{eq:j}) we
have $q(3d_1)\subset \{0,1,2\}$, $q(3d_2)\subset \{ 1,2\}$ and
$q(3d_3)\subset \{0,1\}$ and (\ref{eq:kk+1}) holds for $j=3$.

Hence (\ref{eq:kk+1}) holds and, up to multiplication by some
$\lambda_k$, we have $q(A)\cap \{0,4\}=\emptyset$ as desired.

Suppose now that $A_s=\{ d_1, d_2\}$ and $r_5(d_3) = \pm 2s$ and
assume that (\ref{eq:kk+1}) does not hold  for $j=1$. Without loss
of generality we may assume that $q(d_1)=0$ and $q(d_2)=2$. By
(\ref{eq:j}) we have $q(2d_1)\in \{0,1\}$, $q(2d_2)\in \{0,4\}$. If
(\ref{eq:kk+1}) does not hold for $j=2$ then $q(2d_1)=1$ and
$q(2d_2)=4$. Now, again by (\ref{eq:j}), $q(3d_1)\in \{1,2\}$ and
$q(3d_2)\in \{1,2 \}$ so that (\ref{eq:kk+1}) holds for $j=3$.

Hence we have $q(\lambda_k\cdot A_s)\cap \{ 0,4\}=\emptyset$  at
least for two values of $k$, and since $r_5(d_3) \ne \pm s$ at least
for one of them we have $q(\lambda_k d_3)\neq 0,4$ as well. This
concludes the proof.


\section{Overview of the proof for seven runners}\label{sec:ov}

In what follows,   $m=\max \nu_7(D)$ and $N=7^{m+1}$. We shall omit
the subscript $p=7$ and write $\nu (x)=\nu_7(x)$, $r(x)=r_7(x)$ and
$\Lambda_j=\Lambda_{j,7}$.

  Since we assume that $\gcd (D)=1$, we always have $D_7(0)\neq \emptyset$.
By definition we have
$D_7(m)\neq \emptyset$ as well.

 If $|D_7(i)|\le 3$ for each $0 \le i < m$
then we are done by Corollary \ref{cor:pf1}. Therefore we may
suppose that $|D_7(i)|\ge 4$ for some $i$. On the other hand, if
$|D_7(i_0)|= 4$ for some $i_0>0$ then, again by Corollary
\ref{cor:pf1}, the problem can be reduced to the set $D'=\{
d/p^{i_0}:\; d\in D\setminus D_7(0)\}$. Indeed, if we can find a
multiplier $\lambda'$ such that $|\lambda' d'|_{N'}\ge N'/p$ for
each $d'\in D'$, where  $N'=N/p^{i_0}$,  then $|\lambda d|_N\ge N/p$
for each $d\in D_7(i)$, $i\ge i_0$ with $(\lambda)_N=(\lambda')_N$
and Corollary \ref{cor:pf1} applies. Therefore we only have to
consider the cases $|D_7(0)|=4$ and $|D_7(0)|=5$. These two cases
are dealt with by considering the congruence classes modulo seven of
the elements in $A=D_7(0)$. Since we can replace every element $d\in
A$ by $-d$, we may assume that all elements in $A$ belong to three
nonzero congruence classes modulo $7$, say $(A)_7\subset \{
1,2,4\}$. Let $A_s=\{ d\in A:\; (d)_7=s\} $, $s\in \{ 1,2,4\}$.
Recall that, for $\lambda_k=1+k7^m\in \Lambda_0$ we have
\begin{equation}\label{eq:rot}
q(\lambda_k\cdot A_s)=q(A_s)+ks
\end{equation}

The case $|A|=4$ is simpler and is treated in Section
\ref{sec:d0=4}. The case $|A|=5$ is  more involved   and it is
described in Section \ref{sec:d0=5}. In both cases the general
strategy consists of {\it compressing} the sets $A_s$, $s\in
\{1,2,4\}$ and then using (\ref{eq:rot}). For this we often apply
the Prime Filtering Lemma to subsets of $A-A$ or $2A-2A$.

In what follows we shall denote by $\ell (X)$, where $X$ is a set of
integers, the length of the smallest arithmetic progression of
difference one in $\Z_7$ which contains $q(X)$.

\section{The case $|A|=4$}\label{sec:d0=4}

Let $A=\{ d_1,d_2, d_3, d_4\}\subset D_7 (0)$ and $d_5\in
D_{7}(i_0)$, $0<i_0\le m$. Recall that for any $d \in D_7(m)$ we
have $|\lambda d|_N\ge N/7$. Set $d_5=up^{i_0}$ and let $u'$ such
that $uu'\equiv 1 \pmod{7^{m+1-i_0}}$. Let
$$
\Lambda =\{ju'(1+7^{m-i_0}):\;  1\le j \le 5\}.
$$
For each $\lambda \in \Lambda\cup \Lambda_0$ we clearly have
\begin{equation}\label{eq:inv}
|\lambda d_5|_N=j(7^m+7^{i_0})\ge N/7.
\end{equation}
We shall show that there are $\lambda\in \Lambda_0$ and $\lambda'\in
\Lambda$ such that $q(\lambda\lambda' \cdot A)\cap \{0,6\}
=\emptyset$, thus concluding the case $|A|=4$.

Let $\lambda_k=1+k7^m$, $0\le k\le 6$, denote the elements in
$\Lambda_0$ and  $\lambda_j'=js(1+7^{m-i_0})$, $1\le j\le 5$,  the
ones in $\Lambda$. For $d\in A$ we have
\begin{equation}\label{eq:l1}
q(\lambda_{j+1}'d)\in  q(\lambda_{j}'d)+q(\lambda_1'd)+\{
0,1\}\subset (j+1)q(\lambda_1'd)+\{ 0,1,\ldots ,j\},
\end{equation}
and
\begin{equation}\label{eq:l2}
q(\lambda_kd)=q(d)+kr(d).
\end{equation}

We consider three cases according to the cardinality $|A_s|$ of the
most popular congruence class in $A$.

{\it Case 1.\/} $|A_s|=4$.

If we show that $\ell (\lambda'\cdot A)\le 5$ for some
$\lambda'\in \Lambda$ then, in view of (\ref{eq:l2}), we have $
\{0,6\}\cap q(\lambda_k\lambda'\cdot A)=\emptyset$ for at least
one value of $k$ and we are done.

Suppose this is not the case. Without loss of generality we may then
assume that $q(\lambda_1'\cdot A)=\{0,2,4,6\}$, say
$q(\lambda_1'd_i)=2(i-1)$, $1\le i\le 4$ . In view of (\ref{eq:l1}),
we have $q(\lambda_3' d_i)\in 6(i-1)+ \{0,1,2\}$. Since $\{ 2,3\}
\cap q(\lambda_3'\cdot A)\neq\emptyset$ we must have $q
(\lambda_3'd_1)=3$. Similarly, $\{ 3,4\} \cap q(\lambda_3'\cdot
A)\neq\emptyset$ implies $q (\lambda_3'd_4)=4$. Now, again by
(\ref{eq:l1}),
$$q(\lambda_4' d_1)\in
\{2,3\}, \;\;  q(\lambda_4'd_2)\in \{ 1,2,3,4\},\;\;  q(\lambda_4'
d_3)\in \{2,3,4,5\} \mbox{ and }  q(\lambda_4'd_4)\in \{ 3,4\},
$$
which yields $\{0,6\}\cap q(\lambda_4'\cdot A)=\emptyset$, a
contradiction.

{\it Case 2.\/} $|A_s|=3$.

Let $A_s=\{ d_1,d_2,d_3\}$, so that either $d_4\in A_{2s}$ or
$d_4\in A_{4s}$.

 Suppose that
\begin{equation}\label{eq:3terms}
\ell (\lambda' \cdot A_s)\le 4 \end{equation} for some
$\lambda'\in \Lambda$. Then, in view of (\ref{eq:l2}), we have
$\{0,6\}\cap q(\lambda_{k}\lambda\cdot A_1)=\emptyset$ for $k\in
\{ k_0, k_0+s^{-1}\}$  and some $k_0$ (the values taken modulo
seven). By (\ref{eq:l2}) one of these two values sends $\lambda'
d_4$ outside of $\{0,6\}$ and we are done.

Suppose  that (\ref{eq:3terms}) does not hold. Then we may assume
that either $q(\lambda_1'\cdot A_s)=\{ 0,1,4\}$ or
$q(\lambda_1'\cdot A_s)=\{ 0,2,4\}$, say $q(\lambda_1'd_1)=0$,
$q(\lambda_1'd_2)=1$ or $2$ and $q(\lambda_1' d_3)=4$. If
$q(\lambda_1'd_2)=1$, by (\ref{eq:l1}),
$$q(\lambda_2'\cdot A_s)\subset \{0,2,1\}+\{0,1\}= \{0,1,2,3\},$$ and
(\ref{eq:3terms}) holds, a contradiction. If $q(\lambda_1'd_2)=2$,
using (\ref{eq:l1}) with $\lambda_3'$, we have
$$(q(\lambda_3'd_1),q(\lambda_3'd_2),q(\lambda_3'd_3))\subset \{0,1,2\} \times \{6,0,1\} \times \{5,6,0\}.$$
Since $\{ 2,3,4\}\cap q(\lambda_3'\cdot A_s)\neq \emptyset$ we
have $q(\lambda_3'd_1)=2$, and $\{ 3,4,5\}\cap q(\lambda_3'\cdot
A_s)\neq \emptyset$ implies $q(\lambda_3'd_3)=5$. But then
$q(\lambda_4'd_1)\subset 2+\{0,1\}$ and $q(\lambda_4'd_3)\subset
2+\{0,1\}$, so that
$$q(\lambda_4'\cdot A_s)\subset (2+\{0,1\})\cup\{1,2,3,4\},
$$
and (\ref{eq:3terms}) holds, again a contradiction.

{\it Case 3.\/} $|A_s|=2$.

We may assume that either   $|A_{2s}|=2$ or   $|A_{2s}|=|A_{4s}|=1$.
Let $A_s=\{d_1, d_2\}$.

Suppose that
\begin{equation}\label{eq:2terms}
 \ell (\lambda'\cdot A_s)\le 2
\end{equation}
for some $\lambda' \in \Lambda$. Then we have $\{0,6\}\cap
q(\lambda_{k}\lambda\cdot A_1)=\emptyset$ for $k\in \{ k_0,
k_0+s^{-1}, k_0+2s^{-1}, k_0+3s^{-1}\}$  and some $k_0$ (the
values taken modulo seven). It is a routine checking that for at
least one of these four values of $k$ we have $\{0,6\}\cap
q(\lambda_k\lambda \cdot (A\setminus A_s))=\emptyset$ as well and
we are done.

Suppose that (\ref{eq:2terms}) does not hold. We may assume that
either (i) $q(\lambda_1'\cdot A_s)=\{ 0,2\}$ or (ii)
$q(\lambda_1'\cdot A_s)=\{ 0,3\}$.

Assume that (i) holds. Then
$(q(\lambda_3'd_1),q(\lambda_3'd_2))\subset \{0,1,2\} \times
\{6,0,1\}$. Since (\ref{eq:2terms}) does not hold,
$(q(\lambda_3'd_1),q(\lambda_3'd_2))$ is one of the pairs $(1,6),
(2,0)$ or $(2,6)$. In the two former ones we have
$q(\lambda_4'\cdot A_s)\subset \{1,2\}$ or $q(\lambda_4'\cdot
A_s)\subset \{2,3\}$ respectively, a contradiction; in the last
one, $(q(\lambda_4'd_1),q(\lambda_4'd_2))\subset \{ 2,3\}\times
\{1,2\}$, so that $q(\lambda_4'd_1)= 3$ and $q(\lambda_4'd_1)=1$,
which in turn implies $q(\lambda_5'\cdot A_s) \subset \{3,4\}$,
again a contradiction.

Assume now that (ii) holds. Repeated use of (\ref{eq:l1}) and the
fact that (\ref{eq:2terms}) does not hold gives
\begin{eqnarray*}
(q(\lambda_2'd_1),q(\lambda_2'd_2))\subset \{0,1\}\times
\{6,0\}&\mbox{ implies
}&q(\lambda_2'd_1)=1\mbox{ and }q(\lambda_2'd_1)=6\\
(q(\lambda_3'd_1),q(\lambda_3'd_2))\subset \{1,2\}\times
\{2,3\}&\mbox{ implies
}&q(\lambda_3'd_1)=1\mbox{ and }q(\lambda_3'd_1)=3\\
\end{eqnarray*}
Hence,
$$
q(\lambda_5'\cdot A_s)\subset q(\lambda_2'\cdot
A_s)+q(\lambda_3'\cdot A_s)+\{ 0,1\}=\{2,3\},$$ giving
(\ref{eq:2terms}).  This completes the proof for the case $|A|=4$.


\section{The case $|A|=5$ and $ m >$ 1}\label{sec:d0=5}

Recall that $N=7^{m+1}$ where we now assume that $m=\max(\nu (D))\ge
2$, and  that all elements in $A$ belong to three nonzero congruence
classes modulo $7$, say $(A)_7\subset \{ 1,2,4\}$.   In particular,
given any two elements in $A$ we  have either $r(y)=r(x)$ or
$r(y)=2r(x)$ or $r(x)=2r(y)$.     We find convenient to introduce
the following notation:

\begin{equation}
e(x,y)=\left\{
         \begin{array}{ll}
           2x-y, & \hbox{ if } r(y)=2r(x) \\
           2y-x, & \hbox{ if } r(x)=2r(y) \\
           x-y, & \hbox{ if } r(y)=r(x),
         \end{array}
       \right.
\; \mbox{ and }\; \tilde{e} (x,y)=\left\{
         \begin{array}{ll}
           2q(x)-q(y), & \hbox{ if } r(y)=2r(x) \\
           2q(y)-q(x), & \hbox{ if } r(x)=2r(y) \\
           q(x)-q(y), & \hbox{ if } r(y)=r(x).
         \end{array}
       \right.
\end{equation}

The following properties  can be easily checked.

\begin{lemma}\label{lem:he}
 Let $x, y$ be integers with $\nu (x)=\nu (y)=j<m$.

{\rm (i)} For each $\lambda \in \cup_{i\le j} \Lambda_i$ we have
$\tilde{e}(x,y)=\tilde{e}(\lambda x,\lambda y)$.

{\rm (ii)} $|\tilde{e}(x,y)-q(e(x,y))|_7\le 1.$ Moreover, if
$r(x)=r(y)$ then   $\tilde{e}(x,y)-q(e(x,y)) \in \{0,6\}$.
\end{lemma}

\begin{proof} Let  $\lambda =(1+k7^i)$. If $i<j$ then
$q(\lambda x)=q(x)$ and $q(\lambda y)=q(y)$ so there is nothing to
prove. If $i=j$ and $r(y)=2r(x)$ then $q(\lambda x)=q(x)+kr(x)$
and $q(\lambda y)=q(y)+kr(y)=q(y)+2kr(x)$ so that
$\tilde{e}(\lambda x,\lambda y)=2q(\lambda x)-q(\lambda
y)=2q(x)-q(y)=\tilde{e}(x,y)$. The case $r(y)=r(x)$ can be
similarly checked.

Part (ii) follows directly from the definition of $q(x)=\left(
\lfloor \frac{x}{7^m}\rfloor\right)_7$.
\end{proof}

Let $A_1$ denote the most popular class and denote by $s$ its
congruence class modulo $7$.  Denote by $A_2$ and $A_4$ the subsets
of elements in $A$ congruent with $2s$ and $4s$ modulo $7$
respectively. Recall that, for a subset $X\subset \Z$, $\ell (X)$
stands for the length of the shorter arithmetic progression of
difference $1$ in $\Z_7$ which contains $q(X)$. As in the case
$|A|=4$, the general strategy consists in `compressing' the sets
$q(A_1), q(A_{2}), q(A_{4})$. We summarize in lemmas
\ref{lem:length5} and \ref{lem:length6} below some sufficient
conditions in terms of the values of lengths of these three sets
which allows one to conclude that (\ref{eq:goal}) holds.

\begin{lemma}\label{lem:length5}   Assume that

$$\ell (A_1)+\ell (A_{2})+\ell (A_{4})\le 5.$$

Then there is $\lambda\in \Lambda_0$ such that
$$
q(\lambda\cdot A)\cap \{0,6 \}=\emptyset,
$$
unless  $(\ell (A_1),\ell (A_{2}),\ell (A_{4}))= (3,1,1)$ and
$\tilde{e} (d,d')\in \{2,4\}$ for each $d\in A_{2}$ and $d'\in
A_{4}$.
\end{lemma}

\begin{proof}
The elements of $\Lambda_0$ will be denoted by $\lambda_k=
1+7^mks^{-1}$. Observe that, for $d\in A_{j}$, we have
$q(\lambda_kd)=q(d)+jk$. By (\ref{eq:rot}) we may assume that
$q(A_s)\subset \{1,2,\ldots ,\ell (A_s)\}$, so that
 $q(\lambda_k\cdot A_s)\cap \{0,6\}=\emptyset$ for
$k=0,1,\ldots, 5-\ell (A_s)$. We may assume that $\ell (A_1)+\ell
(A_{2})+\ell (A_{4})= 5.$

If $\ell (A_1)=5$ we are done. If $\ell (A_1)=4$ then we clearly
have $q(\lambda_k\cdot (A\setminus A_1))\cap \{0,6\}=\emptyset$ for
at least one of $k=0,1$.

Suppose that $\ell (A_1)=3$. If either $\ell (A_{2})=2$ or $\ell
(A_{4})=2$ then for at least one of the values of $k=0,1,2$ we have
$q(\lambda_k\cdot (A\setminus A_1))\cap \{0,6\}=\emptyset$. Let us
consider the case $\ell (A_{2})=\ell (A_{4})=1$. Let $q(A_{2})=\{
i\}$ and $q(A_{4})=\{ j\}$. Suppose that $q(\lambda_k\cdot A)\cap \{
0,6\}\neq\emptyset$   for each $k=0,1,2$. Since at most one element
in $\{ i, i+2,i+4\}$ belongs to $\{ 0,6\}$, two of the elements in
$\{ j,j+4,j+1\}$ must be in $\{ 0,6\}$. The only possibility is
$\{j,j+1\}=\{ 0,6\}$ and $\{ i+2\}\in \{0,6\}$. This implies
$2i-j\in \{2,4\}$. Hence $\tilde{e}(d,d')\in \{ 2,4\}$ for each
$d\in A_2, d'\in A_4$.

Suppose finally that $\ell (A_1)=2$. We may assume that $\ell
(A_{2})=2$ and $\ell (A_{4})=1$. Let $q(A_{2})=\{i, i+1\}$ and
$q(A_{4})=\{ j\}$. Two of the four sets $\{i, i+1\}+2k$,
$k=0,1,2,3$, intersect $\{0,6 \}$ for two consecutive values of $k$
in cyclic order. At most two of the sets $\{j\}+4k$ , $k=0,1,2,3$,
intersect $\{ 0,6\}$ for two non consecutive values of $k$. Hence
there is some value of $k$ for which $(q(A_{2})+2k)\cup
(q(A_{4})+4k)$ does not intersect $\{0,6\}$. This completes the
proof.
\end{proof}

\begin{lemma}\label{lem:length6} Suppose that $q(A_1)\subset \{ 1,\ldots ,\ell (A_1)\}$
and let $d\in A_1$ with $q(d)=1$. There is $\lambda\in \Lambda_0$
such that
$$
q(\lambda\cdot A)\cap \{0,6 \}=\emptyset,
$$
if one of the following conditions hold:

{\rm (i)} Either $(\ell (A_1),\ell (A_{2}),\ell (A_{4}))= (5,0,1)$
and $\tilde{e}(d,d')\not\in \{ 4,6\}$ for each $d'\in A_4$, or
$(\ell (A_1),\ell (A_{2}),\ell (A_{4}))= (5,1,0)$ and
$\tilde{e}(d,d')\not\in \{ 2,3\}$ for each $d'\in A_2$.

{\rm (ii)} Either $(\ell (A_1),\ell (A_{2}),\ell (A_{4}))= (4,0,2)$,
or $(\ell (A_1),\ell (A_{2}),\ell (A_{4}))= (4,2,0)$ and
$\tilde{e}(d,d')\neq 4$, where $q(A_2)=\{ i,i+1\}$ and $d'\in
A_2\cap q^{-1}(i)$.

{\rm (iii)} $(\ell (A_1),\ell (A_{2}),\ell (A_{4}))= (3,3,0).$

\end{lemma}
\begin{proof} We may assume that the elements of $A_1$ are congruent
to $1$  modulo $7$, so that
 $q(\lambda_k\cdot A_1)\cap \{0,6\}=\emptyset$ for
$k=0,1,\ldots, 5-\ell (A_1)$.

(i) Suppose that $\ell (A_4)=1$.  If $q(A_4)=i\in \{ 0,6\}$ then
$\tilde{e}(d,d')=2i-1\in \{6, 4\}$. Similarly, if $\ell (A_2)=1$,
then $i=q(A_2)\in \{0,6\}$ implies $\tilde{e}(d,d')=2-i\in \{2,3\}$.

(ii)  Suppose that  $\ell (A_4)=2$, say $q(A_2)=\{i,i+1\}$. One of
the two sets  $\{ i,i+1\}, \{i+4,i+5\}$ does not intersect $\{0,6\}$
so that the result holds for at least one $\lambda_k$, $k=0,1$. If
 $\ell (A_2)=2$ then both $q(A_2)=\{i,i+1\}$ and
$q(\lambda_2\cdot A)=\{ i+2, i+4\}$ intersect $\{ 0,6\}$ only if
$i=5$ and $\tilde{e} (d,d')=2-i=4$.

(iii) Let $q(A_2)=\{ i,i+1,i+2\}$. Now $q(\lambda_k\cdot A_1)\cap
\{0,6\}=\emptyset$ for $k=0,1,2$, and one of the three sets
$\{i,i+1,i+2\}, \{i+2,i+3,i+4\}, \{i+3,i+4,i+5\}$ does not intersect
$\{ 0,6\}$.\end{proof}

As shown in the lemmas \ref{lem:length5} and \ref{lem:length6}
above, compression alone is usually not enough to conclude that
(\ref{eq:goal}) holds.  The next lemmas provide additional tools to
complete the proof. Further results of the same nature will appear
later on in dealing with specific cases.

\begin{lemma}\label{lem:m2}
Let $X\subset \Z_7$ and let $d, d'$ be two integers with $\nu
(d)=\nu(d')=0$ and $r(d')= 2 r(d)$. There is $\lambda\in \Lambda_h$
for some $h<m$ such that
$$
\tilde{e}(\lambda d,\lambda d') \not \in X
$$
whenever one of the two following conditions holds:

{\rm (i)} \;\; $\nu(2d-d')<m$ and $\ell (X)\le 4$, or

{\rm (ii)}\;\; $\nu(2d-d')=m$ and $r(2d-d')\not \in X \cap (X+1)$.
\end{lemma}

\begin{proof}

(i) Since $\ell (X)\le 4$ there is $x \in \Z_7 \setminus (X +\{
0,1,2\})$. Choose $\lambda\in \Lambda_h$, where $h=\nu(2d-d')<m$,
 such that $q(\lambda(2d-d'))=x-1$. Using  Lemma \ref{lem:he}  we
have $\tilde{e}(\lambda d,\lambda d') \in q(e(\lambda d, \lambda
d')+ \{-1, 0, 1\}= x + \{-2,-1,0\} \not \in X$.

(ii) Since $\nu(2d-d')=m$ we have $r(2d-d')=q(2d-d')=q(2d)-q(d')$.
Suppose that $q(2d-d')\not\in X$. Choose $\lambda\in \Lambda_1$ such
that $q(\lambda (7d))=0$. Let us show that this $\lambda$ verifies
the conditions of the lemma (recall that $m>1$). Note that
$q(2\lambda d)\in 2q(\lambda d))+\{ 0,1\}$ and $q(2\lambda
d)=2q(\lambda d))+1$ implies  $q(\lambda (7d))= q(\lambda
(2d+2d+2d+d))\in q(2\lambda d)+q(2\lambda d)+q(2\lambda d)+q(\lambda
d)+\{0,1,2,3\}= 7q(\lambda d) +3 + \{0,1,2,3\}$, contradicting
$q(\lambda (7d))=0$. Hence $q(2\lambda d)= 2q(\lambda d))$. Thus
$\tilde{e}(\lambda d,\lambda d')=
 2q(\lambda d)-q(\lambda d')=q(\lambda (2d-d'))=q(2d-d')\not \in
X$ as claimed. A similar argument applies when $q(2d-d')\not\in
(X+1)$ by choosing $\lambda\in \Lambda_1$ such that $q(\lambda
(7d))=6$ so that $q(\lambda (2d))=2q(\lambda d))+1$.
\end{proof}

Note that the proof of Lemma \ref{lem:m2} (ii) requires $m>1$. We
give a last lemma before starting with the case analysis. First we
note the following remark.

\begin{rem}\label{rem:length}
Let $X \subset \Z$.

{\rm (i)} If $q(X-X)\subset \{0,6\}$ then $\ell (k\cdot X) \le k+1$,
$1\le k\le 6$.

{\rm (ii)} If $q(X-X)\subset   \{0,1,5,6\}$ then $\ell (X) \le 3$.
\end{rem}

\begin{lemma}\label{lem:b3} Let $B=\{ b_1, b_2, b_3\}\subset \Z$ with
$\nu (B)=\{0\}$ and $r(b_1)=r(b_2)=r(b_3)$. Set $x=b_1-b_3$ and
$y=b_2-b_3$.

{\rm (i)} If $\nu (x)\neq \nu (y)$ then there is a multiplier
$\lambda$ such that $\ell (\lambda\cdot B)\le 2$.

{\rm (ii)} If $\nu (x)= \nu (y)=h<m$ and  $r(y)=jr(x)$, $j\in \{
2,3\}$ then there is a multiplier $\lambda$ such that $q(\lambda
y)\in \{0,5,6\}$ and $\ell (\lambda\cdot B)\le j$. Moreover, if
$j=3$, then $\lambda \in \Lambda_h$.
\end{lemma}

\begin{proof} (i) Suppose $\nu(x)> \nu
(y)$. By Lemma \ref{lem:pf} there is a multiplier $\lambda$ such
that $q(\lambda x)=0$ (if $\nu (x)<m$) or $\lambda x=N/7$ (if $\nu
(x)=m$), and $q(\lambda y)=0$. Thus $q(\lambda x-\lambda
y)=q(\lambda (b_1-b_2))\in \{ 0,6\}$ and $q(\lambda\cdot
B-\lambda\cdot B)\in \{ 0,6\}$. By Remark \ref{rem:length} (i), we
have $\ell (B)\le 2$.

(ii) Suppose first that $r(y)=2r(x)$ and set $e=e(x,y)=2x-y$. Choose
$\lambda\in \Lambda_h$ such that either $q(\lambda e)=0$ (if
$\nu(e)\neq m$) or $e=N/7$ (if $\nu(e)= m$). By Lemma \ref{lem:he}
we have $\tilde{e}(\lambda x,\lambda y)\in \{ 0,1,6\}$. Choose
$\lambda'\in \Lambda_h$ such that $q(\lambda'\lambda x)=0$ (if
$\tilde{e}(\lambda x,\lambda y)\in \{0,1\}$) or $q(\lambda'\lambda
x)=6$ (if $\tilde{e}(\lambda x,\lambda y)=6$). Then
$q(\lambda'\lambda y)\in \{0,6\}$ and $q(\lambda'\lambda (y-x))\in
\{0,6\}$. Thus $q(\lambda'\lambda\cdot (B-B))\subset \{ 0,6\}$. By
Remark \ref{rem:length} (i), we have $\ell (\lambda'\lambda\cdot
B)\le 2$.

Suppose now that $r(y)=3r(x)$. Choose $\lambda\in \Lambda_h$ such
that
$$q(\lambda y)=\left\{ \begin{array}{lll}0 &\textrm{if} & 3q(y)+5q(x) \in
\{0,2\}\\ 6 &\textrm{if} & 3q(y)+5q(x) \in \{1,4,6\}\\5 &\textrm{if}
& 3q(y)+5q(x) \in \{3,5\} \end{array} \right.$$ Thus $q(\lambda y)
\in \{0,5,6\}$, $q(\lambda x)\in \{0,1,5,6\}$ and $q(\lambda y)-
q(\lambda x) \in \{0,1,6\}$. By Lemma \ref{lem:he}, $q(\lambda(y-x))
\in q(y)- q(x) + \{0,6\} = \{0,1,5,6\}$. Hence, by Remark
\ref{rem:length} (ii), we have $\ell (\lambda\cdot B)\le 3$.
\end{proof}

\subsection{Case 1: $|A_1|=5$.}\label{subsec:a1=5}

In what follows we shall use some appropriate numbering $\{
d_1,d_2,d_3,d_4,d_5\}$ of the elements in $A$. We shall write
$e_{ij}=e(d_i,d_j)$.

\begin{lemma}\label{lem:n4}
Suppose that $|A_1|\ge 4$   and let $E= (A_1-A_1)\setminus \{ 0\}$.
There is a numbering of the elements of $A_1$ such that one of the
following holds:

{\rm (i)} $\quad \nu(e_{21})>\nu(e_{31})$, or

{\rm (ii)} $\nu (e)=h$ for each $e\in E$ and $r(e_{31})=2r(e_{21})$
and either

\hspace{1cm} {\rm (ii.1)} $r(e_{41})=3r (e_{21})$, or

\hspace{1cm} {\rm (iii.2)} $r(e_{41})=4r(e_{21})$.
\end{lemma}

\begin{proof}
If $|\nu(E)|>1$ then color the pair $\{ d_i,d_j\}$ with
$\nu(e_{ij})$. Two intersecting pairs must have different colors. We
can rename $d_1$ their common element and $d_2, d_3$ the other two
to get (i).

Suppose now that $|\nu (E)|=1$. Consider the set  $X=\{r(e_{il}),
r(e_{jl}), r(e_{kl})\}$ where $i,j,k,l$ are pairwise distinct
subscripts. If $r(e_{il})= r(e_{jl})$ then $\nu
(e_{ij})=\nu(e_{il}-e_{jl})>\nu (e_{il})$ contradicting $|\nu
(E)|=1$. By symmetry the elements in $X$ are pairwise distinct and
two of them belong to one of the sets $\{1,2,4\}$ or $\{3,5,6\}$. We
may thus assume that $(r(e_{il}), r(e_{jl}), r(e_{kl}))=(x,2x,y)$.
If $y\in \{ 3x,4x\}$  then (ii) holds with $l=1, i=2, j=3$ and
$k=4$. If $y\in \{ 5x, 6x\}$ then $(r(e_{lj}), r(e_{ij}),
r(e_{kj}))=(-2x,-x,y-2x)$ and (ii) holds with $j=1, i=2, l=3$ and
$k=4$.
\end{proof}

Let $A_1= \{d_1, d_2, d_3, d_4, d_5\}$ where we use the labeling
provided by Lemma \ref{lem:n4}. We shall show that, up to some
multiplier, we have $\ell (A_1)\le 5$. The result then follows by
Lemma \ref{lem:length5}.

Let $B=\{ d_1, d_2, d_3\}$ and $E=(A_1-A_1)\setminus \{ 0\}$.
Suppose that $\nu (E)\neq \{ m\}$. Then, by Lemma \ref{lem:n4} and
Lemma \ref{lem:b3}, we may assume $\ell (B)\le 2$. Thus, by
(\ref{eq:rot}) we may assume that $q(B) \subset \{0,6\}$.

If $\{q(d_4), q(d_5)\} \ne \{2,4\}$ then we have  $\ell (A_1) \le 4$
and we are done. Suppose $\{q(d_4), q(d_5)\} = \{2,4\}$. Then
$q(3d_4), q(3d_5) \in \{0,1,5,6\}$ and $q(3\cdot B) \subset
\{0,1,2,4,5,6\}$. Furthermore, by Remark \ref{rem:length} (i), $\ell
(3\cdot  B) \le 4$. Thus we have either $q(3\cdot A_1) \subset
\{0,1,2,5,6\}$ or $q(3\cdot A_1) \subset \{0,1,4,5,6\}$, so that
$\ell (3\cdot A_1) \le 5$.

Suppose now that  $\nu (E)\neq \{ m\}$. Set $E_1=\{e_{51}, e_{41},
e_{31}, e_{21}\}\subset \{kN/7, \; 1\le k\le 6\}$. Denote by $iN/7$
and $jN/7$   the elements in the complement of $E_1$. Multiplying by
$(i-j)^{-1} \in U\Z_7$ we may assume that the elements in $E_1$ are
consecutive, so that $\ell (A_1)\le 5$.
 This completes this case.

\subsection{ Case $|A_1|=4$.}\label{subsec:a1=4}

Let $A_1= \{d_1, d_2, d_3, d_4\}$ and $r(d_5) \in \{2s,4s\}$, where
the elements in $A_1$ are labeled with the ordering of Lemma
\ref{lem:n4}.

Let $B=\{ d_1, d_2, d_3\}$ and $E=(A_1-A_1)\setminus \{ 0\}$.
Suppose that $\nu (E)\neq \{ m\}$. Then, by Lemma \ref{lem:n4} and
Lemma \ref{lem:b3} we may assume that
 $\ell (B)\le 2$ and, by (\ref{eq:rot}) we may assume that
$q(B)\subset \{0,6\}$. If $q(d_4)\neq 3$ then $\ell (A_1)\le 4$, and
if $q(d_4)= 3$ then $q(2\cdot A_1)\subset \{ 0,1,5,6\}$ and again
$\ell (A_1)\le 4$. The result follows by Lemma \ref{lem:length5}.

Suppose now that $\nu (E)\neq \{ m\}$.  If  (ii.1) holds then
multiplying by $(r(e_{21}))^{-1}$ (modulo $7$) we may assume that
$r(e_{41})=3N/7$, $e_{31}=2N/7$ and $e_{21}=N/7$ which yields $\ell
(A_1)\le 4$. The result follows by Lemma \ref{lem:length5}. Assume
that (ii.2) holds. Up to multiplication by some $\lambda\in
\Lambda_m$ we may assume that $q(\{ e_{13},e_{24},e_{34}\})\subset
\{ 1,2,4\}$ so that $\ell (A_1)\le 5$. By Lemma \ref{lem:m2} we may
also assume that $\tilde{e}(d_4, d_5)\not\in \{ 4,6\}$. Thus Lemma
\ref{lem:length6} (i) applies if $d_5\in A_4$. Finally, if $d_5\in
A_2$, we may also assume that $\tilde{e}(d_4, d_5)\not\in \{ 2,3\}$
by using again Lemma \ref{lem:m2} unless $e_{45}=3N/7$. In this case
we have $q(\{2e_{14}, 2e_{24}, 2e_{34}\})\subset \{1,4,2\}$, so that
$\ell (A_1)\le 5$, while $2e_45=6N/7$ and so $\tilde{e} (2d_4,
2d_5)\not\in \{ 2,3\}$, and the result also follows from Lemma
\ref{lem:length6} (i). This completes this case.

\subsection{Case  $|A_1|=3$, $|A_2|=1$ and $|A_4|=1$.}\label{subsec:3,1,1}

First we consider a convenient labeling of the elements in $A_1$ to
be used here and the two following subsections.

\begin{lemma}\label{lem:n3}
Let $s\in \Z_7^*$ be given. There is a labeling of the elements in
$A_1$ such that one of the following holds:

{\rm (i)} $\nu(e_{13})>\nu(e_{23})= \nu(e_{12})$,

{\rm (ii)} $\nu(e_{13})=\nu(e_{23})=h$ and either

\hspace*{1cm} {\rm (ii.1)} $r(e_{13})=2r(e_{23})$, or
 {\rm (ii.2)} $r(e_{13})=3r(e_{23})=\pm s$,

%
\end{lemma}

\begin{proof}
Let $E=(A_1-A_1)\setminus \{ 0\}$. If $|\nu(E)|>1$ then we clearly
can label the elements in $D$  to get (i). Assume that $|\nu(E)|=1$.
Suppose that $|r(E)|<6$. Note that we can not have
$r(e_{ij})=r(e_{ik})$ since otherwise
$\nu(e_{kj})=\nu(e_{ij}-e_{ik})>\nu (e_{ij})$ contradicting
$|\nu(E)|=1$. Similarly, $r(e_{ij})=r(e_{kj})$. Thus we may assume
that the repeated values of $r$ on $E$ are $r(e_{ij})=r(e_{jk})$ and
we can label $i=1$, $j=2$ and $k=3$. Suppose now that $|r(E)|=6$.
Thus we may assume that $r(e_{ik})=s$. Observe that $r(e_{jk}) \in
\{3s,5s\}$. Indeed, if $r(e_{jk})=2s$ then
$r(e_{ji})=r(e_{jk})-r(e_{ki})=s$, if $r(e_{jk})=4s)$ then
$r(e_{ij})=r(e_{ik})+r(e_{kj})=4s$ and if $r(e_{jk})=6s$ then
$r(e_{kj})=s$, contradicting in each case  $|r(E)|=6$. If
$r(e_{jk})=3s$ then $r(e_{ji})=-5s $, $r(e_{ki})=-s$ and we can
label $i=3$, $j=2$ and $k=1$. In case $r(e_{jk})=5s$ we can label
$i=1$, $j=2$ and $k=3$.
\end{proof}

Let $A_1=\{d_1, d_2, d_3\}$, where we use the labeling of Lemma
\ref{lem:n3},  $A_{2}= \{d_4\}$ and  $A_{4}=\{d_5\}$. We divide the
proof according to the cases of Lemma \ref{lem:n3}.

(i) and (ii.1) with $h<m$. By Lemma \ref{lem:b3} applied to $A_1$ we
may assume that $\ell (A_1)\le 2$ and the result follows by Lemma
\ref{lem:length5}.

(ii.2) with $h<m$. We consider  two subcases:

(ii.2.a) $\nu (e_{45})\neq \nu (e_{13})$. If $\nu (e_{45})> \nu
(e_{13})$, by Lemma \ref{lem:pf} applied to $e_{45}$    we may
assume that $q(e_{45})=6$ and thus $\tilde{e}(d_4, d_5)\in
\{0,1,5,6\}$. We can then apply   Lemma \ref{lem:b3} to $B=A_1$ so
that  we may assume that $\ell (A_1)\le 3$, yielding the
conditions of Lemma \ref{lem:length5}. A similar argument works
when $\nu (e_{45})< \nu (e_{13})$ by applying Lemma \ref{lem:b3}
first and then Lemma \ref{lem:pf}.

(ii.2.b) $h=\nu (e_{45})= \nu (e_{13})$. By Lemma \ref{lem:n3} we
may assume $r(e_{45})=\pm r(e_{13})$. Suppose first that $r(e_{45})=
r(e_{13})$. Set $f=e_{45}-e_{13}$, so that $\nu (f)>h$. By Lemma
\ref{lem:pf} we may assume $q(f)=0$ (or $f=N/7$ with the same
consequences) so that $\tilde{e}(e_{13},
e_{45})=q(e_{13})-q(e_{45})\in \{0,1\}$. Recall that this last value
is invariant by multiplication of elements in $\Lambda_j$ with $j\le
h$. Put $\tilde{u}= 3q(e_{13})+5q(e_{23})$. By Lemma \ref{lem:pf},
$q(e_{45})$ can be set to the    value shown in the following table
according to the values of $\tilde{e}=\tilde{e}(e_{13}, e_{45})$ and
$\tilde{u}$:

{\small
$$\begin{array}{c|c|c|c|c|c|c|c|c|c|} \multicolumn{2}{r|}{\tilde{u}}
& 0 & 1 & 2 & 3 & \multicolumn{2}{c|}{4} & 5 & 6 \\
\hline & q(e_{45}) & 0 &
6 & 0 &  3 & \multicolumn{2}{c|}{6} & 0 & 6 \\
\cline{2-10} \tilde{e}=0 \quad &
\begin{array}{l}  q(e_{13})
\\ q(e_{23}) \\ q(e_{12})  \end{array} &
\begin{array}{c} 0 \\ 0 \\  \{0,6\} \end{array}&
\begin{array}{c} 6 \\ 5 \\  \{0,1\} \end{array}&
\begin{array}{c} 0 \\ 6 \\  \{0,1\} \end{array}&
\begin{array}{c} 3 \\ 3 \\  \{0,6\} \end{array}&
\multicolumn{2}{c|}{\begin{array}{c} 6 \\ 0 \\  \{6,5\}
\end{array}}&
\begin{array}{c} 0 \\ 1 \\  \{6,5\} \end{array}&
\begin{array}{c} 6 \\ 6 \\  \{0,6\} \end{array} \\
\hline  & q(e_{45}) & 6 & 0 & 6 & 0 & 5 & 6& 6 &
  5\\ \cline{2-10}
\tilde{e}=1 \quad  &
\begin{array}{l}
q(e_{13}) \\ q(e_{23})\\ q(e_{12})  \end{array} &
\begin{array}{c}  0 \\ 0 \\ \{0,6\} \end{array}&
\begin{array}{c}  1 \\ 1 \\ \{0,6\} \end{array}&
\begin{array}{c}  0 \\ 6 \\ \{0,1\} \end{array}&
\begin{array}{c}  1 \\ 0 \\ \{0,1\} \end{array}&
\begin{array}{c}  6 \\ 0 \\ \{6\} \end{array} &
\begin{array}{c}  0 \\ 5 \\ \{1\} \end{array} &
\begin{array}{c}  0 \\ 1 \\ \{5,6\} \end{array}&
\begin{array}{c}  6 \\ 6 \\ \{6,6\} \end{array} \\ \hline
\end{array}$$
}

If $\tilde{u}=4$ and $\tilde{e}=1$ we set $q(e_{45})=5$ if
$q(e_{12})=q(e_{13})-q(e_{23})$, and $q(e_{45})=6$ if
$q(e_{12})=q(e_{13})-q(e_{23})-1$. In all cases except
$\tilde{e}=0$ and $\tilde{u}=3$ we either have $\ell (A_1)\le 3$,
$\ell (A_2)=\ell (A_4)=1$ and $\tilde{e}(d_4, d_5)\not\in \{
2,4\}$ or $\ell (A_1)=2$ and $\ell (A_2)=\ell (A_4)=1$ so that
Lemma \ref{lem:length5} applies. If $\tilde{e}=0$ and
$\tilde{u}=3$ we have  $q(\{ 2e_{45}, 2e_{13}, 2e_{23})\} \in
\{0,6\}$ and $q(2e_{12}) \in \{0,1,5,6\}$ reaching the same
conditions.

A similar analysis applies if $r(e_{45})=-r(e_{13})$ by exchanging
$f= r(e_{13})-r(e_{45})$ by $f= r(e_{13})+r(e_{45})$ and $q(e_{45})$
by $q(-e_{45})$.

(ii.1) and $h=m$. Up to some multiplier in $\Lambda_m$ we may assume
that $e_{13}=2N/7 \quad \textrm{and} \quad e_{23} = N/7$, which
leads to $\ell (A_1) = 3$. By Lemma \ref{lem:m2} we may also assume
that $\tilde{e}(d_4, d_5) \not \in {2, 4}$ and we are in     the
conditions of Lemma \ref{lem:length5}.

(ii.2) and $h=m$. We may assume that $A_1$ contains the elements
congruent to $1$ modulo $7$. Up to some multiplier in $\Lambda_m$ we
may assume that $e_{13}=3N/7 \quad \textrm{and} \quad e_{23} = N/7$
which implies $\ell (A_1)=4$. We may also assume   that
$q(A_1)\subset \{1,2,3,4\}$ and  $q(d_3)=1$. Let $q(d_4)=i$ and
$q(d_5)=j$. In this case (\ref{eq:goal}) holds unless both $\{i,j\}$
and $\{ i+2, j+4\}$ intersect $\{0,6\}$, namely when $i\in \{0,6\}$
and $j\in \{2,3\}$ or $j\in \{0,6\}$ and $i\in \{ 4,5\}$. Thus
$(\tilde{e}_{34}, \tilde{e}_{45})$ is one of the four pairs $
\{(2,4), (2,5), (3, 2), (3,3)\}$ in the first case and one of the
four pairs $\{(4,3), (4,4), (5,1), (5,2)\}$ in the second one.

If $\nu (e_{34})<m$ then by Lemma \ref{lem:m2} (i) applied to $d_3$
and $d_4$ we may assume that $\tilde{e}_{34}\not\in \{2,3,4,5\}$ and
we are done.

 If $\nu (e_{34})=m$ and $\nu(e_{45})<m$ then by applying Lemma \ref{lem:m2} (i) to $d_4$ and
$d_5$ we may assume that $\tilde{e}_{45}\not \in\{2,3,4,5 \}$ if
$\tilde{e}_{34}\in \{ 2,3 \}$ and $ \tilde{e}_{45}\not\in\{1,2,3,4
\}$ if $\tilde{e}_{34} \in \{4,5\}$ thus avoiding the two bad
cases.

Suppose that $\nu (e_{34})=\nu (e_{45})=m$. Observe that one of the
four pairs $(q(e_{34})-\epsilon_1,q(e_{45})-\epsilon_2)$,
$\epsilon_1,\epsilon_2\in \{ 0,1\}$ is not a bad pair. Observe also
that $\tilde{e}(d_3,d_4)=2(q(d_3))-q(d_4)\in
q(2d_3)-q(d_4)+\{0,6\}=q(e_{34})+\{0,6\}$ and similarly
$\tilde{e}(d_4,d_5)\in q(e_{34})+\{0,6\}$. We have $d_4 = 2d_3
+tN/7$, for some $t<7$. By Lemma \ref{lem:pf} we may assume that
$q(7d_3)=4\epsilon_1+2\epsilon_2$ for each choice of $\epsilon_1,
\epsilon_2\in \{ 0,1\}$ (recall that $\nu (7d_3)=1<m$). By a routine
checking we then conclude that
$(\tilde{e}(d_3,d_4),\tilde{e}(d_4,d_5))=
(q(e_{34})-\epsilon_1,q(e_{45})-\epsilon_2)$. Thus each of the eight
bad pairs can be avoided. This completes the proof of this case.

\subsection{Case  $|A_1|=3$ and $|A_2|=2$.}

By using the labeling of Lemma \ref{lem:n3} we have $A_1= \{d_1,
d_2, d_3\}$ and $A_2= \{d_4,d_5\}$. We first prove the following
Lemma.

\begin{lemma}\label{lem:a1=3} Assume that
$|A_1|=3$ and $|A_2|=2$. Let $d \in A_1$ such that $q(A_1) \subset
\{q(d),\dots, q(d)+\ell(A_1)-1\}$ and $d' \in A_2$. If one of the
following conditions hold then there is a multiplier $\lambda$ such
that
$$q(A)\cap \{ 0,6\}=\emptyset.
$$

(i) $\ell (A_1)\le 3$.

(ii) $\ell (A_1)=4$ and $\tilde{e}(d, d')\subset \{0,1,6\}$
\end{lemma}

\begin{proof}
(i) Since $\ell (A_1)\le 3$ there are three good multipliers for
$A_1$ in $\Lambda_0$. At most one of them is bad for each of the two
elements in $A_2$.

(ii) We may assume $q(A_1) \subset\{1,2,3,4\}$. Let
$\lambda_k=7^mks^{-1}\in \Lambda_0$, so that $\lambda_0, \lambda_1$
are good multipliers for $A_1$. Since $\tilde{e}(d, d')\subset
\{0,1,6\}$, then $\lambda_0 d'\in \{1,2,3\}$, and
$\lambda_1d'\in\{3,4,5\}$. At most one of $\lambda_0$, $\lambda_1$
is bad for the second element in $A_2$.
\end{proof}

We divide the proof according to the cases in Lemma \ref{lem:n3}.

(i) or (ii) with $h<m$. By Lemma \ref{lem:b3} applied to $B=A_1$ we
may assume that   $\ell (A_1)\le 3$ and the result follows from
Lemma \ref{lem:a1=3} (i).

(ii.1) with $h=m$. Up to a multiplier in $\Lambda_m$ we may assume
that $e_{13}=2N/7$ and $e_{23} = N/7$ so that $\ell (A_1)=3$. The
result follows from Lemma \ref{lem:a1=3} (i).

(ii.2) with $h=m$. We can apply Lemma \ref{lem:pf} to set
$e_{13}=3N/7$ and $e_{23} = N/7$. Thus $\ell (A_1)=4$.  If $\nu
(e_{34})<m$ then by Lemma \ref{lem:m2} (i) we can set
$\tilde{e}_{34}\not\in \{ 2,3,4,5\}$ and the result follows from
Lemma \ref{lem:a1=3} (ii). We can do the same if r $\nu (e_{35})<m$.
Suppose that $\nu (e_{34})= \nu (e_{35})=m$. Then $\nu (e_{45})=m$.
Set $s=(5r(e_{45}))_7$. By renaming $d_4$ and $d_5$ if necessary we
may assume $e_{45}=e_{23}=N/7$ so that $\ell (A_1)=4$ and $\ell
(A_2)=2$. Moreover, by Lemma \ref{lem:m2}, we may also assume that
$\tilde{e}_{35}\neq 4$. The result follows from Lemma
\ref{lem:length6} (ii).

\subsection{Case $|A_1|=3$ and $|A_4|=2$.}

By using the labeling of Lemma \ref{lem:n3} we have $A_1= \{d_1,
d_2, d_3\}$ and $A_2= \{d_4,d_5\}$. Suppose that, up to a
multiplier,
 \be\label{eq:33}
 \ell (A_1)\le 3\quad \mbox{ and}\quad \ell (A_4)\le 3 \qquad \mbox{ or
 }\qquad \ell (A_1)\le 4\quad \mbox{ and}\quad \ell (A_4)\le 2.
 \ee
Then either Lemma \ref{lem:length5} or Lemma \ref{lem:length6} (iii)
applies.

  We divide the proof according to the cases of Lemma
\ref{lem:n3}.

(i) or (ii.1) with $h<m$. By Lemma \ref{lem:b3} we may assume that
$\ell (A_1)\le 2$. If $\ell (A_4)\le 3$ we are in the conditions of
Lemma \ref{lem:length5}. Otherwise we have $q(e_{45})\in\{2,3,4\}$.
If $|e_{45}|_N \ge 5N/14$ then $|2e_{45}|_N \le 2N/7$  which yields
$\ell (2\cdot A_1)\le 3, \ell (2\cdot A_4)\le 3$ and (\ref{eq:33})
holds. If $|e_{45}|_N \in [2N/7,5N/14]$ then $|3e_{45}|_N \le N/7$
which yields $\ell (3\cdot A_1)\le 4, \ell (3\cdot A_4)\le 2$ and
(\ref{eq:33}) holds.

(ii.2) with $h<m$. If $\nu(e_{13})=\nu(e_{45})$ we choose
$s=r(e_{45})$. By renaming $d_4$ and $d_5$ if   necessary, we may
assume $r(e_{45})= r(e_{13})$. Let $f=e_{45}-e_{13}$, so that $\nu
(f)>\nu (e_{13})$. By Lemma \ref{lem:pf} we may assume $q(f)=0$ (if
$\nu (f)\neq m$) or $f=N/7$ (if $\nu (f)=m$). By Lemma \ref{lem:b3}
(ii) we may also assume that $q(e_{13})\in \{0,5,6\}$ and $\ell
(A_1)\le 3$. Hence, $q(e_{45})=q(e_{13}+f)\in q(e_{13})+q(f)+\{
0,1\}$ which implies $|q(e_{45})|_N\le 2$ and $\ell (A_4)\le 3$.
Therefore (\ref{eq:33}) holds.

Suppose now $\nu(e_{13})\neq \nu(e_{45})$. If $\nu(e_{13})<
\nu(e_{45})$ we can apply Lemma \ref{lem:pf} to $e_{45}$ to set
$q(e_{45})=0$ (if $\nu (e_{45})\neq m$) or $e_{45}=N/7$ (if $\nu
(e_{45})= m$)  and then Lemma \ref{lem:b3} to $A_1$ to set $\ell
(A_1)\le 3$ and $\ell (A_4)\le 2$ and (\ref{eq:33}) holds. A
similar argument applies if $\nu(e_{13})> \nu(e_{45})$ by applying
Lemma \ref{lem:b3} first and then Lemma \ref{lem:pf}.

(ii.1) with $h=m$. We may assume that $e_{13}=2N/7 \quad
\textrm{and} \quad e_{23} = N/7$ so that $\ell (A_1) = 3$. By
renaming $d_4$ and $d_5$ if necessary we may assume $r(e_{45}) \in
\{1,2,4\}$. We consider two cases.

(a)  Either $\nu(e_{45})< m$ or $e_{45} \in \{N/7, 2N/7\}$. By Lemma
\ref{lem:pf} we may assume $q(e_{45})\in \{0,1,5,6\}$ so that $\ell
(A_4)\le 3$ yielding (\ref{eq:33}).

(b)  $\nu(e_{45})= m$ and $e_{45} \not\in \{N/7, 2N/7\}$. We may
then assume that $e_{45}= 4N/7$,  $q(A_1)= \{1,2,3\}$ and
$q(A_4)=\{i,i+4\}$. There are three available multipliers in
$\Lambda_0$ for which $q(\lambda\cdot A_1)\cap \{0,6\}=\emptyset$.
It can be easily checked that one of them verifies $q(\lambda\cdot
A_4)\cap \{0,6\}$ as well unless $i\in \{2,6\}$, and so,
$\tilde{e}(3,4) \in \{4,5\}$. By Lemma \ref{lem:m2} we may assume
that $\tilde{e}(d_1,d_4)$ takes none of these two values unless
$e_{45}= 5N/2$. If this is the case, we have $\ell(2\cdot A_1)=5$,
$\ell(2\cdot A_4)=1$, $e_{34}=3N/7$ and, by Lemma \ref{lem:m2}(ii),
we can avoid $\tilde{e}_{34} =2 $. By (\ref{eq:rot}) we may assume
$q(A_1)= \{1,3,5\}$ and, since $\tilde{e}_{34}=3$, we have
$q(A_4)=\{1,2\}$.

(ii.2) with $h=m$. Choose $s=r(e_{45})$. By exchanging $d_4$ with
$d_5$ if necessary we may assume that $r(e_{13})=r(e_{45})$, and by
Lemma \ref{lem:pf} we may assume that $e_{13}=N/7$ and
$e_{23}=3N/7$, so that $\ell (A_1)\le 4$. If $\nu (e_{45})=m$ we
have $e_{45}=N/7$ and $\ell (A_4)=2$. If $\nu (e_{45})<m$ then, by
Lemma \ref{lem:pf} we may set $q(e_{45})\in \{0,6\}$ and $\ell
(A_4)=2$ again. In both cases Lemma \ref{lem:length6} (ii) applies.

\subsection{Case $|A_1|=2$ .}\label{subsec:a1=2}

We may assume that $A_1= \{d_1, d_2\}$, $A_{2}= \{d_3, d_4\}$ and
$A_{4}=\{d_5\}$.

Up to renaming the elements in $A$ we may assume that $r(e_{12}),
r(e_{34})\in \{1,2,4\}$. Suppose that \be\label{eq:221} \ell
(A_1)\le 2\quad \mbox{ and }\quad \ell (A_2)\le 2.\ee Then the
result follows from Lemma \ref{lem:length5}.

If $\nu (e_{12})\neq \nu (e_{34})$ then, by Lemma \ref{lem:pf}
applied to $e_{12}$ and $e_{34}$ we may assume that $q(e_{12}),
q(e_{34})\in \{0,6\}$. Hence (\ref{eq:221}) holds.

Assume now that $\nu (e_{12})= \nu (e_{34})$. Suppose first that
$r(e_{12}) = r(e_{34})$. Let $f= e_{12}- e_{34}$. Note that
$\nu(e_{12})<\nu (f)$.  If $\nu (f)\le m$,  by Lemma \ref{lem:pf}
applied to $f$ and $e_{34}$  we may assume that $q(f)=6$ and
$q(e_{34})=0$. On the other hand, if $\nu (f)> m$, so that $q(f)=0$,
we can apply Lemma \ref{lem:pf} to $e_{34}$ to have $q(e_{34})=6$.
In both cases, Lemma \ref{lem:he} yields $q(e_{12})=q(f+e_{34})\in
\{0,6\}$ and (\ref{eq:221}) holds.

Suppose now that  $r(e_{12}) \ne r(e_{34})$. Then either
$r(e_{12})=2r(e_{34})$ or $2r(e_{12})=r(e_{34})$. If $\nu
(e_{12})<m$ then, we can set $e_{12}$ and $e_{34}$ in $\{0,6\}$ by Lemma
\ref{lem:b3} and  the desired multiplier exists.  Assume $\nu(e_{12}) = \nu(e_{34})= m$. We consider
two cases:

(a) $r(e_{12})= 2r(e_{34})$. Up to a multiplier in $\Lambda_m$ we
may assume that $e_{12}=2N/7$ and $e_{34} = N/7$. Let  $\lambda_k
\in \Lambda_0$ be such that $q(\lambda_k d_2)=k$, $k= 1,2,3$.  Since
$e_{12}=2N/7$ we have $q(\lambda_k\cdot A_{1})=\{ k, k+2\}$. On the
other hand, since $q(d_4)=\tilde{e}(d_2,d_4)-2q(d_2)$ and
$e_{34}=N/7$, we have $q(\lambda_k\cdot A_2)=\{\tilde{e}(d_2,d_4)-2k,
\tilde{e}(d_2,d_4)-2k+1)\}$. Finally, $q(d_5)=
4\tilde{e}(d_2,d_5)+4k$. Thus, if
\be\label{eq:badpairs}(\tilde{e}(d_2,d_4),\tilde{e}(d_2,d_5))\not
\in \{((4,2),(4,4),(5,4),(6,4),(6,6)\}\ee then $q(\lambda_k\cdot
A)\cap \{0,6\}=\emptyset$ for some $k$. Now, if either
$\nu(e_{24})<m$, by using Lemma \ref{lem:m2}(i), or $e_{24}\in
\{N/7, 2N/7, 3N/7\}$ we can assume that $\tilde{e}(d_2,d_4) \not \in
\{4,5,6\}$ and so (\ref{eq:badpairs}) holds; if $\nu(e_{24})=m$ and
$\nu(e_{25})<m$ we can avoid each of the pairs in
(\ref{eq:badpairs}) by setting $\tilde{e}(d_2,d_5) \not \in \{2,4\}$
if $e_{24} \in \{4N/7,5N/7\}$ and $\tilde{e}(d_2,d_5) \not \in
\{4,6\}$ if $e_{24} \in \{0,6N/7\}$; finally, if
$\nu(e_{24})=\nu(e_{25})=m$,  we observe that one of the four pairs
$(q(e_{34})-\epsilon_1,q(e_{45})-\epsilon_2)$,
$\epsilon_1,\epsilon_2\in \{0,1\}$ avoids each of the bad pairs in
(\ref{eq:badpairs}). By  repeating the argument of case
\label{subsec:3,3,1} applied to $7d_5$, (\ref{eq:badpairs}) holds.

(b) $r(e_{12})= 4r(e_{34})$. Up to a multiplier in $\Lambda_m$ we
may assume that $e_{12}=N/7$ and $e_{34} = 2N/7$.  Let  $\lambda_k
\in \Lambda_0$ be such that $q(\lambda_k d_4)=k$, $k= 1,2,3$, so
that $q(\lambda_k\cdot A_2)=\{ k, k+2\}$. On the other hand,
$q(d_2)=4\tilde{e}(d_2,d_4)+4k$, so that $q(\lambda_k\cdot A_1)=\{
4\tilde{e}(d_2,d_4)+4k, 4\tilde{e}(d_2,d_4)+4k+1\}$. Finally
$q(d_5)=2q(d_2)-\tilde{e}(d_2,d_5)=\tilde{e}(d_2,d_4)-\tilde{e}(d_2,d_5)+k$.
By Lemma \ref{lem:m2} we can assume that $\tilde{e}(d_2,d_4) \not
\in \{2,4, 6\}$. It is then routine to check that, for every value
of $\tilde{e}(d_2,d_4)$ and $\tilde{e}(d_2,d_5)$ there is $k\in \{
1,2,3\}$ such that $q(\lambda_k\cdot A)\cap \{ 0,6\}=\emptyset$.

This completes the proof of the case $|A|=5$ and $m>1$.

\section{The case $|A|=5$ and $ m =1$}\label{sec:m=1}

In Section \ref{sec:d0=5} we have used the hypothesis $m>1$ in Lemma
\ref{lem:m2} and in particular situations in cases 6.3 and 6.6.
However, when $m=1$ we are led to consider the problem with $N=49$
and $d_6=k7$, $1\le k \le 3$, which is more efficiently handled by
an exhaustive search. There are at most $(1/21){21\choose 5}=969$
non equivalent choices for sets of cardinality $5$ in $U\Z_{49}$. By
computer search we found that there is always a multiplier for which
each of these sets can be included in the interval $[7,42]$ except
in the three (up to dilation) following ones:
$$
\{1,4,11,39,43\}\;\; \{1,4,18,22,29\} \mbox{ and } \{1,4,18,44,46\}.
$$
We consider each of this sets in $\Z_{N'}$ with $N'=2N=98$. There
are at most $32$ nonequivalent subsets in $U\Z_{98}$ which are
congruent to one of the above exceptional sets, and each of them has
to be combined with the six possible values of $d_6$, namely $7k, \;
k=1,2,\ldots ,6$. By checking all these possibilities, we found that
there is always a suitable multiplier except for the sets
\begin{eqnarray*}
\{4,50,60,88,92\}&\mbox{ with } &d_6 \in \{14,28,42\}\\
\{4,18,22,50,78\}&\mbox{ with } &d_6 \in \{14,28,42\}\\
\{4,18,44,46,50\}&\mbox{ with } &d_6 \in \{14,28,42\}.
\end{eqnarray*}
Since $N'$ is an even number and $\gcd(D)\ge 2$, none of these sets
can arise from one of the exceptions found for $N=49$. This
completes the proof.

\end{document}